\documentclass{amsart}
\usepackage{graphicx} 
\usepackage[utf8]{inputenc}
\usepackage{amsmath}
\usepackage{amssymb}
\usepackage{booktabs}
\usepackage{amsmath}
\usepackage{mathtools}
\usepackage{pgfplots}
\usepackage{multicol}
\usepackage{enumerate}
\usepackage{bookmark}
\usepackage{algpseudocode}
\usepackage{xcolor}
\usepackage{amsthm}
\usepackage[english]{babel}
\usepackage[T1]{fontenc}
\usepackage{microtype}
\usepackage[utf8]{inputenc}
\usepackage{hyperref}
\usepackage[utf8]{inputenc}
\usepackage{fourier}
\usepackage{microtype}
\usepackage{hyperref}
\usepackage{comment}
\usepackage{graphicx}
\usepackage{amsmath}
\usepackage{amsthm}
\usepackage{amssymb}
\usepackage{pgfplots}
\usepackage{enumitem}
\usepackage{verbatim}
\usepackage{bookmark}
\usepackage{float}
\usepackage{algpseudocode}
\usepackage{xcolor}
\usepackage{mathtools} 
\usepackage{pdfpages}
\usepackage{tikz}
\usepackage{tikz-cd}
\usepackage{setspace}
 \usepackage[capitalise]{cleveref}
\usepackage{todonotes}
\usepackage{fancyhdr}
\title{On the Hereditariness of the Representations of Thread Quivers}
\author{Enrico Maria Del Regno}
\address{Mathematics Research Centre, Tampere University, Finland}
\email{enrico.delregno@tuni.fi}
\makeatletter
\renewcommand{\section}{\@startsection{section}{1}%
  \z@{.7\linespacing\@plus\linespacing}{.5\linespacing}%
  {\centering\normalfont\LARGE\bfseries}}

\renewcommand{\subsection}{\@startsection{subsection}{2}%
  \z@{.5\linespacing\@plus.7\linespacing}{.5\linespacing}%
  {\normalfont\Large\bfseries}}
\makeatother

\newtheorem{thm}{Theorem}[section]

\newtheorem{defn}[thm]{Definition}
\newtheorem{rem}[thm]{Remark}

\newtheorem{lemma}[thm]{Lemma}
\newtheorem{prop}[thm]{Proposition}
\newtheorem{corollary}[thm]{Corollary}

\newcommand{\Ext}{{\rm Ext}}

\newcommand{\rep}{{\rm rep}}

\newcommand {\Hom}{{\rm Hom}}

\def\Vec_F{{\rm Vect_{\mathbb{F}}}}
\def\vec_F{{\rm vect_{\mathbb{F}}}}

\begin{document}
\maketitle
\begin{abstract}
We prove a conjecture of Paquette, Rock, and Yildirim by showing that, for every thread quiver, the abelian category of pointwise finite dimensional representations is hereditary. Since this category typically lacks enough projectives and injectives, standard homological methods do not apply directly. Our approach combines a Yoneda $\Ext$ criterion for hereditariness, established in this paper, with structural reductions to the subcategory of quasi noise free representations. We also indicate an alternative proof using a Keller's theorem on derived categories.
\end{abstract}
\section*{Introduction}

The primary goal of this paper is to resolve a conjecture proposed in October 2024 by Paquette, Rock, and Yildirim \cite[Conjecture 7.5]{thread}, proved in Theorem \ref{hered}. It states that for every thread quiver, the abelian category of pointwise finite dimensional representations is hereditary, meaning that $\Ext^2(-,-)=0$. 

The concept of a thread quiver was originally introduced by Berg and Van Roosmalen in \cite{thread2}, who considered only discrete linear orders as threads. However, we adopt the recent generalization by Paquette, Rock, and Yildirim in \cite{thread}: a thread quiver is a standard quiver where each arrow is equipped with an arbitrary linearly ordered set, called a thread. The associated thread category $C$ is built by inserting the elements of each thread as intermediate objects between the source and the target of the corresponding arrow. The morphisms in $C$ are generated by the original arrows and the order relations along the threads. 

The study of thread quivers is motivated by the persistent homology pipeline in Topological Data Analysis (TDA). While classical TDA typically computes homology over discrete steps, modern applications increasingly require continuous, multiparameter filtrations. Thread quivers are the algebraic framework needed for this shift: by replacing discrete arrows with linearly ordered sets (threads), they model representations that evolve continuously. For a comprehensive background on standard homology pipeline, continuous and multiparameter persistence modules, see the survey paper by Botnan and Lesnick \cite{survey}.

Let $\mathbb{K}$ be a field. A representation of $C$ over $\mathbb{K}$ is a covariant functor from $C$ to the category of $\mathbb{K}$-vector spaces. In the classical setting of a finite quiver $Q$, the category $\mathrm{Rep}(Q)$ of all representations is well known to be hereditary (see, e.g., \cite{assem}). However, as shown by Paquette, Rock, and Yildirim in \cite{thread}, this fundamental homological property is lost for thread categories, where the entire category $\mathrm{Rep}(C)$ fails to be hereditary. 

A similar difference appears with the Krull--Schmidt decomposition. This theorem holds for the entire category $\mathrm{Rep}(Q)$ when $Q$ is a finite quiver. Instead, for a thread category $C$, it only works within the subcategory of pointwise finite-dimensional representations, $\mathrm{rep}^{\mathrm{pwf}}(C)$. Since $\mathrm{rep}^{\mathrm{pwf}}(C)$ recovers this essential structural property, it provides the natural setting to investigate whether hereditariness is also restored.

\textbf{Contributions and Proof Strategy.}
To prove our main result, we must first overcome the problem that $\mathrm{rep}^{\mathrm{pwf}}(C)$ may not have enough projectives or injectives. 

\begin{itemize}
\item \textbf{ General Yoneda $\Ext$ criterion (Theorem~\ref{hered condition}):} We show that for every abelian category $\mathcal{C}$ and every $n\geq 1$, fixing one argument, the Yoneda functor $\Ext^n_{\mathcal{C}}$ is right exact if and only if $\Ext^{n+1}_{\mathcal{C}}(-,-)=0$. For $n=1$, this gives a completely general characterization of the hereditary property (Corollary~\ref{hered corol}) without relying on projective or injective resolutions.
\end{itemize}
We also need to relate $\mathrm{rep}^{\mathrm{pwf}}(C)$ to the subcategory 
$\mathrm{rep}^{\mathrm{qnf}}(C)$ of quasi noise free representations, i.e., those whose Krull--Schmidt decomposition contains only finitely many 
indecomposable noise summands on each fixed thread (where a ``noise'' summand is one whose support does not intersect the original vertices of the quiver). Indeed, by Paquette, Rock, and Yildirim, we know that this subcategory is abelian and the functor 
$\Ext^1_{\mathrm{rep}^{\mathrm{qnf}}(C)}(M,-)$ is right exact for every object $M$ which, by our general $\Ext$ criterion, implies that $\mathrm{rep}^{\mathrm{qnf}}(C)$ is hereditary.
\begin{itemize}
    
\item \textbf{Structural reductions to $\mathrm{rep}^{\mathrm{qnf}}(C)$ (Theorems~\ref{C2} and \ref{comb prop}):} The main technical step of the paper consists of two reduction results, which provide a key mechanism to deduce the hereditariness of $\mathrm{rep}^{\mathrm{pwf}}(C)$ from that of $\mathrm{rep}^{\mathrm{qnf}}(C)$. More precisely, Theorem~\ref{C2} establishes that any short exact sequence in $\mathrm{rep}^{\mathrm{pwf}}(C)$ starting with a quasi noise free term $B$ allows us to extract a new short exact sequence entirely contained within $\mathrm{rep}^{\mathrm{qnf}}(C)$, which is linked to the original one via a commutative diagram. The same theorem will also present a dual result. Using this reduction inductively, Theorem~\ref{comb prop} provides an explicit construction to prove that for any $X,Z \in \mathrm{rep}^{\mathrm{qnf}}(C)$, there is an isomorphism
\[
\Ext^n_{\mathrm{rep}^{\mathrm{pwf}}(C)}(X,Z) \cong \Ext^n_{\mathrm{rep}^{\mathrm{qnf}}(C)}(X,Z).
\]
This equivalence acts as a crucial step for our main result, allowing us to move extension computations from $\mathrm{rep}^{\mathrm{pwf}}(C)$ into the simpler environment of $\mathrm{rep}^{\mathrm{qnf}}(C)$. Moreover, Theorem~\ref{C2} also yields a more abstract alternative proof of this isomorphism via a theorem of Keller on derived categories.

\item \textbf{Resolution of \cite[Conjecture 7.5]{thread} (Theorem~\ref{hered}):} The main result of this paper is the resolution of the open conjecture from \cite{thread} regarding the hereditariness of $\mathrm{rep}^{\mathrm{pwf}}(C)$. In Theorem~\ref{hered}, we establish this by directly verifying that $\Ext^2_{\mathrm{rep}^{\mathrm{pwf}}(C)}(X,Y) = 0$ for all representations $X$ and $Y$. This is achieved by combining our Yoneda criterion with the structural reduction results of Theorems~\ref{C2} and \ref{comb prop}.

\end{itemize}
\textbf{Organization of the paper.}
Section \ref{1} recalls the basic definitions and constructions concerning thread categories and their representations. In Section \ref{2}, we review the Yoneda $\Ext$  and prove the general Yoneda $\Ext$ criterion for arbitrary abelian categories. Finally, in Section \ref{3}, we establish the main structural reduction results and use them to prove the conjecture.

\textbf{Acknowledgments.}
I am deeply grateful to my supervisor, Eero Hyry, for his constant support and crucial advice during the preparation of this work. Furthermore, I would like to thank the Vilho, Yrjö ja Kalle Väisälän foundation for generously funding this research.

\section{Preliminaries: Quivers and thread categories}\label{1}
\begin{defn}[Quiver]
We let $Q = (Q_0, Q_1, s, t)$ denote any quiver (or oriented graph), possibly infinite. The set $Q_0$ is the set of \textbf{vertices}, and $Q_1$ is the set of \textbf{arrows}. The maps $s, t : Q_1 \to Q_0$ are the \textbf{source} and \textbf{target} maps, respectively.
\end{defn}

\begin{defn}[Thread Quiver and Thread]
Let $Q$ be a quiver. For each arrow $\alpha \in Q_1$, we assign a linearly ordered set $P_\alpha$ (called \textbf{thread}), possibly empty, whose elements are regarded as intermediate points inserted between the source and the target of $\alpha$. The family $P\coloneqq\{P_\alpha\}_{\alpha\in Q_1}$
determines a \textbf{thread quiver}, denoted by $(Q,P)$. 
\end{defn}
\begin{defn}[Thread category]
Let $(Q,P)$ be a thread quiver. For each arrow $\alpha\in Q_1$, let $\overline{P_\alpha}$ be the linearly ordered set obtained from $P_\alpha$ by adjoining two formal elements $\min(\alpha)$ and $\max(\alpha)$, with $\min(\alpha)\leq x\leq \max(\alpha)$ for every $x\in P_\alpha$. Consider the surjective map
\[
\varphi:\coprod_{\alpha\in Q_1}\overline{P_\alpha}\to Q_0\bigcup \left(\coprod_{\alpha\in Q_1}P_\alpha\right),
\]
which is the identity on $\coprod_{\alpha\in Q_1}P_\alpha$ and such that $\varphi(\min(\alpha))=s(\alpha)$ and $\varphi(\max(\alpha))=t(\alpha)$. We define the path category $C$ of $(Q,P)$, called simply the \textbf{thread category}, as follows.
\begin{enumerate}
    \item The objects of $C$ are the equivalence classes of $\coprod_{\alpha\in Q_1}\overline{P_\alpha}$ with respect to the relation $x\sim y$ if and only if $\varphi(x)=\varphi(y)$.
    \item For each $x,y\in \overline{P_\alpha}$ with $x\leq y$, we introduce a morphism $\eta_{y,x}:x\to y$. If $\alpha$ is not a loop, we identify $\eta_{\max(\alpha),\min(\alpha)}$ with $\alpha$.
    \item The morphisms of $C$ are paths generated by these morphisms inside the threads and by the arrows of $Q$, and composition is given by concatenation of paths, subject to the relation $\eta_{z,y}\circ\eta_{y,x}=\eta_{z,x}$ whenever $x\leq y\leq z$ in some $\overline{P_\alpha}$.
\end{enumerate}
For more details on this construction and for examples, we refer to \cite{thread}.
\end{defn}
Given a small category $C$, a \textbf{representation} of $C$ over a field $k$ is a covariant functor $M \colon C \to \mathrm{Mod}(k)$. We denote the category of all such representations by {$\mathrm{Rep}(C)$}. In this work, we focus on its full subcategory $\mathrm{rep}^{\mathrm{pwf}}(C)$ consisting of \textbf{pointwise finite dimensional (pwf) representations}, namely those for which the vector space $M(x)$ is finite dimensional for every object $x \in C$. 

When $C$ is a thread category, we can classify a representation $M$ according to its support $\mathrm{supp}(M)$, i.e., the set of objects $x \in C$ such that $M(x) \neq 0$. A representation $M \in \mathrm{rep}^{\mathrm{pwf}}(C)$ is called \textbf{noise} if its support does not intersect the set of vertices $Q_0$, and \textbf{noise free} if it does not contain non-zero noise direct summands. 

A fundamental property of $\mathrm{rep}^{\mathrm{pwf}}(C)$ is that it is a Krull--Schmidt category, which allows us to reduce many homological arguments to the study of indecomposable summands.

\begin{thm}\cite[Theorem 1.15]{thread} \label{KS}
Every representation $M \in \mathrm{rep}^{\mathrm{pwf}}(C)$ admits a unique decomposition (up to isomorphism and permutation) into a direct sum of indecomposable representations. Consequently, $M$ decomposes uniquely as
\[
M \cong M_{NF} \oplus M_N,
\]
where $M_{NF}$ is noise free and $M_N$ is a noise representation. Furthermore, $M_N$ admits a unique decomposition along the threads:
\[
M_N = \bigoplus_{\alpha \in Q_1} M_{N,\alpha}, \quad \text{with} \quad M_{N,\alpha} = \bigoplus_{j \in I_\alpha} N_j,
\]
where $\mathrm{supp}(M_{N,\alpha}) \subseteq P_\alpha$ for each $\alpha \in Q_1$, and each $N_j$ is an indecomposable noise summand.
\end{thm}

A representation $M \in \mathrm{rep}^{\mathrm{pwf}}(C)$ is called \textbf{quasi noise free (qnf)} if, in its unique Krull--Schmidt decomposition, only finitely many indecomposable noise summands occur on any single thread $P_\alpha$. We denote by $\mathrm{rep}^{\mathrm{qnf}}(C)$ the full subcategory of $\mathrm{rep}^{\mathrm{pwf}}(C)$ consisting of all quasi noise free representations. 

As shown in \cite[Proposition 7.1]{thread}, $\mathrm{rep}^{\mathrm{qnf}}(C)$ is a Serre subcategory of $\mathrm{rep}^{\mathrm{pwf}}(C)$; that is, the class of quasi noise free representations is closed under taking subobjects, quotients, and extensions. Being a Serre subcategory of an abelian category, $\mathrm{rep}^{\mathrm{qnf}}(C)$ is itself abelian.

The tool used to establish this homological property is the notion of a \textbf{valid partition}. A partition $\mathfrak{P}$ of the objects of the thread category $C$ is called valid if each vertex $x \in Q_0$ forms a singleton cell and, for every arrow $\alpha \in Q_1$, the thread $P_\alpha$ is partitioned into finitely many interval cells. This concept, introduced in \cite[Section 2.2]{thread}, will also be crucial for the proof of Theorem~\ref{C2}.

Given a representation $M \in \mathrm{rep}^{\mathrm{pwf}}(C)$, one can associate a \textbf{canonical partition} $\mathfrak{P}_M$, characterized by the property that $x \leq y$ lie in the same cell if and only if they are in the same thread and $M(\eta_{yx})$ is an isomorphism. By \cite[Proposition 2.7]{thread}, the canonical partition $\mathfrak{P}_M$ is valid if and only if $M$ is quasi noise free. In particular, this equivalence guarantees that every quasi noise free representation admits a finite partition into intervals along each thread on which all structural maps are isomorphisms, providing the structural control needed for our main results.

We conclude the preliminaries by recalling the following key result from \cite{thread}.

\begin{thm} \cite[Theorem~7.2]{thread} \label{thm:qnf_hered}
Let $C$ be a thread category. In the category $\mathrm{rep}^{\mathrm{qnf}}(C)$, the functor $\Ext^1_{\mathrm{rep}^{\mathrm{qnf}}(C)}(M,-)$ is right exact for every object $M$.
\end{thm}

In \cite{thread}, this property is taken as the definition of a hereditary category. While a priori this might seem like a weaker condition, we will formally establish in Corollary~\ref{qnf hered} that it is indeed equivalent to hereditariness in the usual homological sense (i.e., $\Ext^2_{\mathcal{C}}(-,-)=0$).
\section{A General Criterion for Hereditary Categories} \label{2}

In this section, we generalize an important criterion for $\Ext$ using the Yoneda definition of the $\Ext$ groups. We use this approach because the category $\mathrm{rep}^{\mathrm{pwf}}(C)$ usually lacks enough projectives or injectives, meaning we cannot use classical derived functor techniques.

We briefly recall the structure of the Yoneda $\Ext$ and refer to \cite[Chapter VII]{book} for a deeper explanation. For any abelian category $\mathcal{C}$ and objects $A, C \in \mathcal{C}$, the group $\Ext^n(C,A)$ is the set of equivalence classes of $n$-extensions with end terms $A$ and $C$. These are the exact sequences of length $n$ of the form:
\[ E: 0 \longrightarrow A \longrightarrow B_1 \longrightarrow \dots \longrightarrow B_n \longrightarrow C \longrightarrow 0. \]
Two $n$-extensions $E$ and $F$ are equivalent if there is an extension $G$ with the same end terms such that they are connected by morphisms $E \longrightarrow G \longleftarrow F$ or $E \longleftarrow G \longrightarrow F$ that are the identity on the end objects $A$ and $C$.

Moreover, we explicitly recall the Yoneda product (or splicing). Given an $n$-extension $E \in \Ext^n(C,A)$ and an $m$-extension $F \in \Ext^m(D,C)$ of the form:
\[ E: \quad 0 \longrightarrow A \longrightarrow B_1 \longrightarrow \dots \longrightarrow B_n \xrightarrow{\lambda} C \longrightarrow 0 \]
\[ F: \quad 0 \longrightarrow C \xrightarrow{\mu} B'_1 \longrightarrow \dots \longrightarrow B'_m \longrightarrow D \longrightarrow 0 \]
their Yoneda product is the $(n+m)$-extension $E \circ F \in \Ext^{n+m}(D,A)$ obtained by concatenating the sequences at $C$ via the composition $\mu\lambda$:
\[ E \circ F: \quad 0 \longrightarrow A \longrightarrow B_1 \longrightarrow \dots \longrightarrow B_n \xrightarrow{\mu\lambda} B'_1 \longrightarrow \dots \longrightarrow B'_m \longrightarrow D \longrightarrow 0. \]

With these tools and the Yoneda $\Ext$ long exact sequence (detailed in \cite[VII.2.2 and VII.5.1]{book}), we can generalize the following homological property to any abelian category.
\begin{thm}\label{hered condition}
Let $\mathcal{C}$ be an abelian category and let $n \geq 1$.
Then the following conditions are equivalent:
\begin{enumerate}
    \item the functor $\Ext_{\mathcal{C}}^n(M,-)$ is right exact for each object $M$ in $\mathcal{C}$;
    \item $\Ext_{\mathcal{C}}^{n+1}(M,A)=0$ for each pair of objects $(A,M) \in \mathcal{C}$.
\end{enumerate}
\end{thm}
\begin{proof}
    \textbf{$1\implies 2$} Let us consider an arbitrary exact sequence $\epsilon\in \Ext^{n+1}(M,A)$.
    \[\epsilon\coloneqq\;\; 0\rightarrow A\xrightarrow f E_1\xrightarrow g E_2\xrightarrow h E_3 \rightarrow\dots  \rightarrow E_{n+1}\xrightarrow i M\rightarrow 0\]
   Now, let us define $K\coloneqq{\rm Im}(g)={\rm Ker}(h)$ and the exact sequences \[\epsilon_1\coloneqq\;\; 0\rightarrow A\xrightarrow fE_1\xrightarrow s K\rightarrow 0 \;\;\;\;\text{and} \;\;\;\; \epsilon_2\coloneqq0\rightarrow K \xrightarrow t E_2\xrightarrow h E_3 \rightarrow \dots  \rightarrow E_{n+1}\xrightarrow i M\rightarrow 0\] where $s$ is the restriction of $g$ considering just its image as codomain and $t$ is the standard inclusion. It follows that $\epsilon=\epsilon_1\circ \epsilon_2$.

Let us consider the long exact sequence for $\Ext(M,-)$ associated to the short exact sequence $\epsilon_1$:
\[ \dots \longrightarrow \Ext^n(M, E_1) \xrightarrow{\phi_s} \Ext^n(M, K) \xrightarrow{\delta_{\epsilon_1}} \Ext^{n+1}(M, A) \longrightarrow \dots \;. \]
By our assumption, the functor $\Ext^n(M,-)$ is right exact, which means the map $\phi_s$ is surjective. Therefore, the image of $\phi_s$ is the entire group $\Ext^n(M, K)$. By the exactness of the sequence, the kernel of $\delta_{\epsilon_1}$ must also be the entire group $\Ext^n(M, K)$. This strictly forces the map $\delta_{\epsilon_1}$ to be zero.

As detailed in \cite[VII.2.2 and VII.5.1]{book}, the connecting homomorphism $\delta_{\epsilon_1}$ in this sequence is given by the Yoneda product (splicing) with $\epsilon_1$. This means that for any element $y \in \Ext^n(M, K)$, we have $\delta_{\epsilon_1}(y) = \epsilon_1 \circ y = 0$. Applying this to our $n$-extension $\epsilon_2 \in \Ext^n(M, K)$, we get:
\[ \epsilon = \epsilon_1 \circ \epsilon_2 = \delta_{\epsilon_1}(\epsilon_2) = 0. \]
By the arbitrariness of $\epsilon$, the claim follows.

\textbf{$2\implies 1$} This implication is immediate by the $\Ext$ long exact sequence. Let us consider the short exact sequence $0 \to A \to B \to C \to 0$ and related long exact sequence associated to $\text{Ext}(M, -)$:
$$\dots \to \text{Ext}^n(M, B) \xrightarrow{\beta} \text{Ext}^n(M, C) \xrightarrow{\delta} \text{Ext}^{n+1}(M, A)\rightarrow\dots \;.$$ 
By assumption, $\Ext^{n+1}(M,A)=0$, hence $\delta=0$. Therefore,
$\operatorname{Ker}(\delta)=\Ext^n(M,C)$.
By exactness,\[\operatorname{Im}(\beta)=\operatorname{Ker}(\delta)=\Ext^n(M,C),\]
so $\beta$ is surjective. Hence $\Ext^n(M,-)$ is right exact.
\end{proof}
\begin{rem} \label{hered rem}
Dually, using the long exact sequence in the first variable, one gets the analogous statement for the contravariant functors $\Ext^n(-,M)$. Namely,
    \[ \Ext_{\mathcal{C}}^n(-,M)\text{ is right exact for each object } M \in \mathcal{C}\iff\]\[ \Ext_{\mathcal{C}}^{n+1}(A,M)=0 \text{ for each pair of objects } (A,M) \in \mathcal{C}.\]
\end{rem}
\begin{corollary} \label{hered corol}
    Given an abelian category $\mathcal{C}$, the following conditions are equivalent:
\begin{itemize}
    \item $\Ext_\mathcal{C}^1(-,M)$ is right exact for any object $M$.
    \item $\Ext_\mathcal{C}^1(M,-)$ is right exact for any object $M$.
    \item $\mathcal{C}$ is hereditary, i.e., $\Ext_\mathcal{C}^2(-,-)=0$.
\end{itemize}
\end{corollary}
\begin{proof}
   The claim follows immediately by applying Theorem \ref{hered condition} and Remark \ref{hered rem} with $n=1$.
\end{proof}
As an immediate application of this characterization, we obtain the hereditary property for the category of quasi noise free representations mentioned in the preliminaries.

\begin{corollary} \label{qnf hered}
Let $C$ be a thread category. The category $\mathrm{rep}^{\mathrm{qnf}}(C)$ of quasi noise free representations is hereditary, meaning that $\Ext^2_{\mathrm{rep}^{\mathrm{qnf}}(C)}(-,-) = 0$.
\end{corollary}
\begin{proof}
The statement follows immediately by combining Theorem \ref{thm:qnf_hered} and Corollary \ref{hered corol}.
\end{proof}
We conclude this section by showing a powerful consequence of the biadditivity of the Yoneda product: the global vanishing of $\Ext$ at any given degree strictly forces the vanishing of all higher extension groups.

\begin{prop} \label{prop:ext_vanish_general}
Let $\mathcal{C}$ be an abelian category and let $n \geq 1$ be an integer. If  $\;\Ext^n_{\mathcal{C}}(A, B) = 0$ for all objects $A, B \in \mathcal{C}$, then for any objects $X, Y \in \mathcal{C}$ and any integer $m > n$, we have
\[ \Ext^m_{\mathcal{C}}(X, Y) = 0. \]
\end{prop}

\begin{proof}
Let $m > n$ and take any $m$-extension class $[E] \in \Ext^m_{\mathcal{C}}(X, Y)$. 

We can factor $E$ at the $n$-th intermediate image, obtaining an intermediate object $K \in \mathcal{C}$ and expressing the class as a Yoneda splice:
\[ [E] = [E_n] \circ [E_{m-n}], \]
where $[E_n] \in \Ext^n_{\mathcal{C}}(K, Y)$ and $[E_{m-n}] \in \Ext^{m-n}_{\mathcal{C}}(X, K)$.

By our assumption, the group $\Ext^n_{\mathcal{C}}(K, Y)$ is trivial, so we have $[E_n] = 0$. Since the Yoneda product is biadditive with respect to the Baer sum (see \cite[Chapter VII, Lemma 3.2]{book}), we can compute the product explicitly:
\[ [E] = 0 \circ [E_{m-n}] = (0 + 0) \circ [E_{m-n}] = 0 \circ [E_{m-n}] + 0 \circ [E_{m-n}]. \]

Subtracting $0 \circ [E_{m-n}]$ from both sides in the abelian group $\Ext^m_{\mathcal{C}}(X, Y)$, we conclude that $[E] = 0$. Thus, every $m$-extension is equivalent to zero.
\end{proof}
\section{Hereditariness of Pointwise Finite Dimensional Representations}\label{3}
The primary objective of this section is to establish the central result of the paper: the hereditariness of the category of pointwise finite dimensional representations of a thread category.
\begin{thm} \label{hered}
    Let $C$ be a thread category. Then $\rep^{\rm pwf}(C)$ is a hereditary category. In particular, for any objects $X, Y \in \mathrm{rep}^{\mathrm{pwf}}(C)$ and any integer $n \geq 2$, we have
    \[ \Ext^n_{\mathrm{rep}^{\mathrm{pwf}}(C)}(X, Y) = 0. \]
\end{thm}
The proof relies on the following preliminary results, which establish the strong connection between the category of pointwise finite dimensional representation and the category of quasi noise free representations. 

\subsection{Connections between  $\mathrm{rep}^{\rm pwf}(C)$ and $\mathrm{rep}^{\rm qnf}(C)$ }

This subsection presents our core results. We begin with Theorem \ref{C2}, the main piece of our strategy, which builds a short exact sequence in $\mathrm{rep}^{\rm qnf}(C)$ from any short exact sequence in $\mathrm{rep}^{\rm pwf}(C)$ whose first (or last) term is in $\mathrm{rep}^{\rm qnf}(C)$. Using this reduction inductively, we obtain Theorem \ref{comb prop}, which allows us to consider intermediate terms of $n$-extensions with quasi noise free ends entirely within $\rep^{\mathrm{qnf}}(C)$. Additionally, the first part of Theorem \ref{C2} gives us a second way to prove Theorem \ref{comb prop} using the derived category approach explained in Remark \ref{rem:derived_keller}, providing a less constructive but shorter proof of Theorem \ref{hered}.

\begin{thm} \label{C2}
Let $C$ be a thread category, $\mathcal A=\mathrm{rep}^{\mathrm{pwf}}(C)$, and $\mathcal B=\mathrm{rep}^{\mathrm{qnf}}(C)$. The full abelian subcategory $\mathcal B\subseteq \mathcal A$ satisfies the following conditions:
\begin{itemize}
\item for each short exact sequence $0 \longrightarrow B \xrightarrow{i} A \xrightarrow{\pi} A' \longrightarrow 0$
in $\mathcal A$ with $B\in\mathcal B$, there exists a commutative diagram
\[
\begin{tikzcd}
0 \arrow[r] & B \arrow[r,"i"] \arrow[d,equal] & A \arrow[r,"\pi"] \arrow[d,"p"] & A' \arrow[r] \arrow[d,"q"] & 0 \\
0 \arrow[r] & B \arrow[r,"i'"] & B' \arrow[r] & B'' \arrow[r] & 0
\end{tikzcd}
\]
where the second row is a short exact sequence in $\mathcal B$. Moreover, $B'$ is a direct summand of $A$.
\item for each short exact sequence $0 \rightarrow  A' \xrightarrow f A \xrightarrow{g} B \rightarrow 0$ in $\mathcal{A}$ with $B \in \mathcal{B}$, there exists a commutative diagram
\[
\begin{tikzcd}
0 \arrow[r] & A' \arrow[r, "f"] \arrow[d, "h"] & A \arrow[r, "g"] \arrow[d, "p"] & B \arrow[r] \arrow[d, equal] & 0 \\
0 \arrow[r] & B'' \arrow[r, "f'"] & B' \arrow[r, "g'"] & B \arrow[r] & 0
\end{tikzcd}
\]
where the second row is a short exact sequence in $\mathcal{B}$. Moreover $B'$ is a direct summand of $A$.
\end{itemize}
\end{thm}
\begin{proof}[Proof of Theorem \ref{C2}]
\textbf{Part 1.} By Theorem~\ref{KS}, the representation $A$ decomposes into noise free and noise parts as $A \cong A_{NF}\oplus A_N$. We can further decompose the noise parts along the threads as $A_N=\bigoplus_{\alpha\in Q_1} A_{N,\alpha}$. Moreover, for each thread $\alpha$, we write $A_{N,\alpha}=\bigoplus_{j\in I_\alpha} N_j$
for its decomposition into indecomposable noise summands. For each arrow $\alpha$ and each $j \in I_\alpha$, let $p_j \colon A \to N_j$ denote the canonical projection.

Since $B\in \mathrm{rep}^{\mathrm{qnf}}(C)$, the associated canonical partition $\mathfrak P_B$ is valid. Fix a thread $\alpha$, and let $\mathfrak P_{B,\alpha}$
be the restriction of $\mathfrak P_{B}$ to the thread $P_\alpha$. Let $m$ be the number of cells in $\mathfrak{P}_{B,\alpha}$. Choose a single sample point from each cell, yielding a finite set of points $x_1, \dots, x_m$. For each chosen point
$x_k$, the vector space $A(x_k)$ is finite dimensional, meaning that for only finitely many noise summands $N_j(x_k)$ is non-zero. Therefore, the linear map $p_{j,x_k} \circ i_{x_k} \colon B(x_k) \to N_j(x_k)$ is non-zero for only finitely many indices $j \in I_\alpha$. We denote this finite set by
\[ J_k\coloneqq\{ j \in I_\alpha \mid p_{j,x_k} \circ i_{x_k} \neq 0 \}. \]
Let $J_\alpha\coloneqq \bigcup_{k=1}^m J_k$. Applying this construction to all threads, we define
\[
B'\coloneqq A_{NF}\oplus\Bigl(\bigoplus_{\alpha\in Q_1}\bigoplus_{j\in J_\alpha} N_j\Bigr).
\]
By construction, $B'$ is a direct summand of $A$ and 
$B'\in\mathrm{rep}^{\mathrm{qnf}}(C)$. Let $p\colon A\to B'$ be the 
canonical projection and $i'\coloneqq p\circ i$.

We claim $i'$ is a monomorphism. If $x$ is a vertex of $Q$, then all noise summands vanish at $x$, so $B'(x)=A_{NF}(x)=A(x)$. Hence $p_x=\mathrm{id}_{A(x)}$, and therefore $i'_x=i_x$, which is injective. Now, let $x$ lie in a thread $P_\alpha$ and let $v\in B(x)$ be nonzero. Let $x_k$ be the sample point of the 
cell of $\mathfrak{P}_{B,\alpha}$ containing $x$. By definition of the 
canonical partition, the structural map of $B$ between $x$ and $x_k$ 
is an isomorphism.

Assume first that $x\leq x_k$: let $w\coloneqq B(\eta_{x_k,x})(v)\neq 0$. By naturality, 
\[i'_{x_k}(w)=B'(\eta_{x_k,x})(i'_x(v)).\] By the construction of $J_k$, 
 $i(B(x_k))\subset B'(x_k)$, so the projection $p_{x_k}$ is the identity.
Therefore, $i'_{x_k}(w)=i_{x_k}(w)$ which is non-zero because $i$ is a monomorphism. Hence 
$i'_x(v)\neq 0$.

Assume now that $x_k \leq x$. Let $w \in B(x_k)$ be such that 
$B(\eta_{x, x_k})(w) = v$; this exists because the structural maps are isomorphisms inside the same cell. By naturality of $i$,
\[
i_x(v) = A(\eta_{x, x_k})(i_{x_k}(w)).
\]
We claim that for any index $j \notin J_k$, the projection of $i_x(v)$ onto $N_j(x)$ is zero. By the naturality of the projection $p_j$, it commutes with the structural maps, yielding:
\[
p_{j,x}(i_x(v)) = p_{j,x}\big(A(\eta_{x, x_k})(i_{x_k}(w))\big) = N_j(\eta_{x, x_k})\big(p_{j,x_k}(i_{x_k}(w))\big).
\]
By the definition of $J_k$, if $j \notin J_k$, then the map $p_{j,x_k} \circ i_{x_k}$ is identically zero. Therefore, $p_{j,x_k}(i_{x_k}(w)) = 0$, which forces also $p_{j,x}(i_x(v))$ to be $0$. This proves that $i_x(v) \in B'(x)$. Since $p_x \colon A(x) \to B'(x)$ is the canonical projection onto $B'(x)$, it acts as the identity on $i_x(v)$, yielding:
\[
i'_x(v) = p_x(i_x(v)) = i_x(v).
\]
Since $i_x$ is injective and $v \neq 0$, it follows that $i'_x(v) \neq 0$. We have shown $i'_x$ is injective for every $x$, so $i'$ is a monomorphism. 

Setting $B''$ as $\operatorname{Coker}(i')$, we obtain the short exact sequence
\[
0 \longrightarrow B \xrightarrow{i'} B' \longrightarrow B'' \longrightarrow 0.
\]
Since $\mathrm{rep}^{\mathrm{qnf}}(C)$ is a Serre subcategory of $\mathrm{rep}^{\mathrm{pwf}}(C)$ (and thus closed under quotients), $B''$ is also in $\mathrm{rep}^{\mathrm{qnf}}(C)$. By the universal property of $A'$ as the cokernel of the first row, there is a unique morphism $q\colon A' \to B''$ making the diagram commute. This completes the proof of the first part.

\textbf{Part 2.} The construction of $B'$ follows a strategy similar to the first part.  As before, we consider the decomposition $A \cong A_{NF}\oplus\big( \bigoplus_{\alpha\in Q_1} A_{N,\alpha}\big)$, where each noise part further decomposes into $A_{N,\alpha} = \bigoplus_{j\in I_\alpha} N_j$. For each $j \in I_\alpha$, let $\iota_j \colon N_j \hookrightarrow A$ denote the canonical inclusion.

For each thread $\alpha$, let $x_1, \dots, x_{m_\alpha}$ be the sample points obtained by picking one point inside each cell of the canonical partition $\mathfrak{P}_{B}$. For each chosen point $x_k$, the space $A(x_k)$ is finite-dimensional. Consequently, the set of indices 
\[  J_k\coloneqq\{ j \in I_\alpha \mid g_{x_k} \circ \iota_{j,x_k} \colon N_j(x_k) \to B(x_k) \text{ is not the zero map}\}\] is finite.
Setting $J_\alpha \coloneqq \bigcup_{k=1}^{m_\alpha} J_k$ for each arrow $\alpha$ , we construct the following subrepresentation of$A$:
\[
B' \coloneqq A_{NF} \oplus \Bigl( \bigoplus_{\alpha \in Q_1} \bigoplus_{j \in J_\alpha} N_j \Bigr).
\]
By construction, $B'$ is a direct summand of $A$ and $B' \in \mathrm{rep}^{\mathrm{qnf}}(C)$. Let $\iota \colon B' \hookrightarrow A$ be the canonical inclusion and define $g' \coloneqq g \circ \iota \colon B' \to B$.

We claim $g'$ is an epimorphism. If $x$ is a vertex of $Q$, all noise summands vanish at $x$, so $B'(x) = A_{NF}(x) = A(x)$. Thus, $g'_x = g_x$, which is surjective. 

Let us show now that $g'_{x_k}=g_{x_k}|_{B'(x_k)}$ is surjective for all sample points. Given $b_k \in B(x_k)$, by the surjectivity of $g_{x_k}$ there exists $a \in A(x_k)$ such that $g_{x_k}(a) = b_k$. We can decompose $a = a' + c$, where $a' \in B'(x_k)$ and $c$ is the sum of the components in $N_j(x_k)$ for $j \notin J_k$. By the definition of $J_k$, the map $g_{x_k}$ vanishes on $N_j(x_k)$ for all $j \notin J_k$, yielding $g_{x_k}(c) = 0$. Therefore, $g'_{x_k}(a') = g_{x_k}(a') = g_{x_k}(a) = b_k$, which proves the surjectivity.

Now, consider an arbitrary point $x$ on a thread $P_\alpha$, and let $b \in B(x)$. Let $x_k$ be the sample point of the cell of $\mathfrak{P}_{B,\alpha}$ containing $x$; the structural map of $B$ between $x$ and $x_k$ is an isomorphism.

Assume first that $x_k \leq x$ and let $b\in B(x)$. Let $b_k \coloneqq B(\eta_{x, x_k})^{-1}(b) \in B(x_k)$, which exists because the structural maps are isomorphisms inside the same cell. By the step above, there exists $a_k \in B'(x_k)$ with $g'_{x_k}(a_k) = b_k$. Set $a \coloneqq A(\eta_{x, x_k})(a_k)$. Since $B'$ is a subrepresentation of $A$, it is invariant under the structural maps of $A$; therefore, $a_k \in B'(x_k)$ immediately forces $a \in B'(x)$. By the naturality of $g$, we have:
\[
g'_x(a) = g_x(a) = g_{x}\circ A(\eta_{x, x_k})(a_k)=B(\eta_{x, x_k})\circ(g_{x_k}(a_k)) = B(\eta_{x, x_k})(b_k) = b.
\]
This proves that $g'_x$ is surjective in this case.

Assume now that $x \leq x_k$. By the surjectivity of $g_x$, there exists $a \in A(x)$ such that $g_x(a) = b$. We can decompose this element as $a = a' + c$, where
\[ a' = a_{NF} + \sum_{j \in J_\alpha} a_j \in B'(x) \quad \text{and} \quad c = \sum_{j \notin J_\alpha} a_j. \]
We claim $g_x(c) = 0$. For each $j \notin J_\alpha$, the element $a_j$ belongs to $N_j(x)$. By the naturality of $g$ applied to the inclusion $\iota_j$, we have:
\[
B(\eta_{x_k, x})(g_x(a_j)) = g_{x_k}\big(A(\eta_{x_k, x})(a_j)\big) = (g_{x_k} \circ \iota_{j,x_k})\big(N_j(\eta_{x_k, x})(a_j)\big).
\]
By the definition of $J_k$ (and since $j \notin J_\alpha \supseteq J_k$), the map $g_{x_k} \circ \iota_{j,x_k}$ is identically zero. Therefore, the right-hand side is zero, yielding $B(\eta_{x_k, x})(g_x(a_j)) = 0$. Since $B(\eta_{x_k, x})$ is an isomorphism, this forces $g_x(a_j) = 0$. Summing over all $j \notin J_\alpha$ gives $g_x(c) = 0$, from which we conclude:
\[ g'_x(a') = g_x(a') = g_x(a) = b. \]

We have shown that $g'_x$ is surjective for every $x$, so $g'$ is an epimorphism. Setting $B'' \coloneqq \ker(g')$, we obtain the short exact sequence:
\[
0 \longrightarrow B'' \xrightarrow{f'} B' \xrightarrow{g'} B \longrightarrow 0.
\]
Since $\mathrm{rep}^{\mathrm{qnf}}(C)$ is a Serre subcategory of $\mathrm{rep}^{\mathrm{pwf}}(C)$ and is therefore closed under subobjects, $B'' \in \mathrm{rep}^{\mathrm{qnf}}(C)$. By the universal property of $A'$ as the kernel of the first row, there is a unique morphism $h \colon A' \to B''$ making the diagram commute. This completes the proof.
\end{proof}

\begin{thm}\label{comb prop}
 Let $C$ be a thread category, $\mathcal{A} = \mathrm{rep}^{\mathrm{pwf}}(C)$, $\mathcal{B} = \mathrm{rep}^{\mathrm{qnf}}(C)$ and $n \geq 1$. Let $X,Z$ be two objects in $\mathcal{B}$. Every $n$-extension $E \in \Ext_{\mathcal{A}}^n(X, Z)$ is equivalent to an $n$-extension $F$ where all terms belong to $\mathcal{B}$. In particular, for each $n\geq1$, one has
\[
\Ext_{\mathcal{A}}^n(X, Z)=\Ext_{\mathcal{B}}^n(X, Z).
\]
\end{thm}

\begin{proof}
We proceed by induction on $n$.

\textbf{Base case ($n=1$).} This is immediate, since, as recalled in the preliminaries, $\mathrm{rep}^{\mathrm{qnf}}(C)$ is a Serre subcategory of $\mathrm{rep}^{\mathrm{pwf}}(C)$, and therefore it is closed under extensions.

\textbf{Inductive step.} Assuming the result holds for $n-1$, we now prove it for $n$. Let $E$ be an $n$-extension in $\mathrm{rep}^{\mathrm{pwf}}(C)$ of the form
\[
E: 0 \to Z \to A_1 \to A_2 \rightarrow A_3 \to \dots \to A_n \to X \to 0.
\]
We factor $E$ at the first intermediate image, $K_1=\operatorname{Im}(A_1\to A_2)$, into the Yoneda splice $E=E_1\circ E'$, where $E_1$ is the short exact sequence
\[
E_1: 0 \to Z \to A_1 \to K_1 \to 0,
\]
and $E'$ is the $(n-1)$-extension starting from $K_1$,
\[
E': 0 \to K_1 \to A_2 \rightarrow A_3 \to \dots \to A_n \to X \to 0.
\]
Applying Theorem~\ref{C2} to $E_1$, we obtain the commutative diagram
\begin{equation}\label{first}
\begin{tikzcd}
E_1:\;0 \arrow[r] & Z \arrow[r] \arrow[d,equal] & A_1 \arrow[r] \arrow[d] & K_1 \arrow[r] \arrow[d,"q"] & 0 \\
F_1:\; 0 \arrow[r] & Z \arrow[r] & B_1 \arrow[r] & K'_1 \arrow[r] & 0
\end{tikzcd}
\end{equation}
where $F_1: 0 \to Z \to B_1 \to K'_1 \to 0$ is a short exact sequence in $\mathrm{rep}^{\mathrm{qnf}}(C)$.
We now construct a new $(n-1)$-extension by pushing out $E'$ along the morphism $q\colon K_1\to K'_1$. Let $\widetilde{A}_2$ be the corresponding pushout. Since pushouts preserve cokernels in abelian categories, this yields an exact sequence
\[
\widetilde{E}': 0 \to K'_1 \to \widetilde{A}_2 \to A_3 \to \dots \to A_n \to X \to 0.
\]
This construction yields a morphism of $(n-1)$-extensions where $q$ acts on the leftmost terms:
\begin{equation}\label{second}
\begin{tikzcd}
E': & 0 \arrow[r] & K_1 \arrow[r] \arrow[d, "q"] & A_2 \arrow[r] \arrow[d] & A_3 \arrow[r] \arrow[d, equal] & \dots \arrow[r] & X \arrow[r] \arrow[d, equal] & 0 \\
\widetilde{E}': & 0 \arrow[r] & K'_1 \arrow[r] & \widetilde{A}_2 \arrow[r] & A_3 \arrow[r] & \dots \arrow[r] & X \arrow[r] & 0
\end{tikzcd}
\end{equation}
Since the morphism $q$ is at the end of Diagram~\ref{first} and at the beginning of Diagram~\ref{second}, we can splice the two commutative diagrams together. This shows that the Yoneda product of the top rows is equivalent to the Yoneda product of the bottom rows, yielding $[E] = [E_1] \circ [E'] = [F_1] \circ [\widetilde{E}']$.
Now notice that $\widetilde{E}'$ is an $(n-1)$-extension whose end terms, $K'_1$ and $X$, both belong to $\mathrm{rep}^{\mathrm{qnf}}(C)$. By the inductive hypothesis, $\widetilde{E}'$ is equivalent to an $(n-1)$-extension $F'$ all of whose terms belong to $\mathrm{rep}^{\mathrm{qnf}}(C)$. Therefore, $E$ is equivalent to the splice $F_1\circ F'$. Since both $F_1$ and $F'$ lie entirely in $\mathrm{rep}^{\mathrm{qnf}}(C)$, their Yoneda product is an $n$-extension with all terms in $\mathrm{rep}^{\mathrm{qnf}}(C)$. This completes the induction.

Alternatively, one could also argue by splitting the $n$-extension from the right. By applying the second part of Theorem \ref{C2} to the short exact sequence ending in $X$, one can construct an equivalent extension in $\mathrm{rep}^{\mathrm{qnf}}(C)$ by using, in the inductive step, a pullback instead of a pushout.
\end{proof}

The following remark shows an alternative way to reach the same result of the theorem above.
\begin{rem}[Alternative proof of Theorem~\ref{comb prop} via derived categories]\label{rem:derived_keller}
While the proof of Theorem \ref{comb prop} given above is constructive, it can be seen as a consequence of a result of Keller. Indeed, because of Theorem \ref{C2}, \cite[Theorem~12.1]{kell} implies that the canonical inclusion functor between the derived categories $D(\mathrm{rep}^{\mathrm{qnf}}(C))$ and $D(\mathrm{rep}^{\mathrm{pwf}}(C))$
is fully faithful. Hence, for every $X,Z\in \mathrm{rep}^{\mathrm{qnf}}(C)$ and $n\geq 1$, we obtain a natural isomorphism
\[
\Hom_{D(\mathrm{rep}^{\mathrm{pwf}}(C))}(X,Z[n])
\cong
\Hom_{D(\mathrm{rep}^{\mathrm{qnf}}(C))}(X,Z[n]).
\]

Now, by \cite[Lemma~13.27.5]{stacks} for any abelian category $\mathcal A$ one has
\[
\Ext^n_{\mathcal A}(X,Z)\cong \Hom_{D(\mathcal A)}(X,Z[n]).
\]
Applying this to $\mathrm{rep}^{\mathrm{pwf}}(C)$ and $\mathrm{rep}^{\mathrm{qnf}}(C)$, and combining it with the full faithfulness of the inclusion functor, it follows that
\[
\Ext^n_{\mathrm{rep}^{\mathrm{pwf}}(C)}(X,Z)
\cong
\Ext^n_{\mathrm{rep}^{\mathrm{qnf}}(C)}(X,Z).
\]
This is exactly the conclusion of Theorem~\ref{comb prop}.

We maintain the explicit approach as our main argument because the constructive proof is of independent interest: it shows how to replace the intermediate terms of an $n$-extension in $\mathrm{rep}^{\mathrm{pwf}}(C)$ with objects from $\mathrm{rep}^{\mathrm{qnf}}(C)$. This explicit reduction mechanism is lost in the derived category framework. 
\end{rem}
\subsection{Products and coproducts in $\rep^{\mathrm{pwf}}(C)$}
In this subsection, we establish two homological properties of $\mathrm{rep}^{\mathrm{pwf}}(C)$ needed to prove Theorem~\ref{hered}. The first lemma allows us to reduce $\Ext$ computations from arbitrary representations to their individual indecomposable summands. The second establishes that products and coproducts coincide in $\mathrm{rep}^{\mathrm{pwf}}(C)$ whenever they are well-defined.

\begin{lemma} \label{lemma ext}
Let $C$ be a thread category. Let $\{X_i\}_{i \in I}$ and $\{Y_j\}_{j \in J}$ be families of objects in $\mathrm{rep}^{\mathrm{qnf}}(C)$ such that $\bigoplus_{i \in I} X_i$ and $\prod_{j \in J} Y_j $ are $ \mathrm{rep}^{\mathrm{pwf}}$. Then, there is a natural isomorphism:
$$ \Ext^n_{\mathrm{rep}^{\mathrm{pwf}}(C)}\Big(\bigoplus_{i \in I} X_i, \prod_{j \in J} Y_j\Big) \cong \prod_{i \in I} \prod_{j \in J} \Ext^n_{\mathrm{rep}^{\mathrm{pwf}}(C)}(X_i, Y_j). $$
\end{lemma}

\begin{proof}
 By standard properties of derived categories (see \cite[Lemma 13.27.5]{stacks}), we can express the $\Ext^n$ group as a morphism vector space in the derived category $D(\mathrm{rep}^{\mathrm{pwf}}(C))$:
$$ \Ext^n_{\mathrm{rep}^{\mathrm{pwf}}(C)}\Big(\bigoplus_{i \in I} X_i, \prod_{j \in J} Y_j\Big) \cong \mathrm{Hom}_{D(\mathrm{rep}^{\mathrm{pwf}}(C))}\Big(\bigoplus_{i \in I} X_i, \prod_{j \in J} Y_j[n]\Big). $$

Recall that limits and colimits of complexes are computed degreewise in the homotopy category and, by \cite[Proposition 10.2.8]{kashiwara}, they are preserved by the localization functor to the derived category. Therefore, using the standard properties about the relations between (co)products and homomorphisms, we have the isomorphisms:
\[\mathrm{Hom}_{ D(\mathrm{rep}^{\mathrm{pwf}}(C))} \left( \bigoplus_{i \in I} A_i, X \right) \cong \prod_{i \in I} \mathrm{Hom}_{ D(\mathrm{rep}^{\mathrm{pwf}}(C))}(A_i, X), \]
\[ \mathrm{Hom}_{ D(\mathrm{rep}^{\mathrm{pwf}}(C))} \left( X, \prod_{j \in J} B_j \right) \cong \prod_{j \in J} \mathrm{Hom}_{ D(\mathrm{rep}^{\mathrm{pwf}}(C))}(X, B_j).\] 
Applying these general isomorphisms, $\mathrm{Hom}_{D(\mathrm{rep}^{\mathrm{pwf}}(C))}\Big(\bigoplus_{i \in I} X_i, \prod_{j \in J} Y_j[n]\Big)$ decomposes naturally as:
\[ \prod_{i \in I} \prod_{j \in J} \mathrm{Hom}_{D(\mathrm{rep}^{\mathrm{pwf}}(C))}(X_i, Y_j[n]). \]
Finally, using the standard identification between morphisms in the derived category and $\Ext$ groups once more, this space is naturally isomorphic to
$$ \prod_{i \in I} \prod_{j \in J} \Ext^n_{\mathrm{rep}^{\mathrm{pwf}}(C)}(X_i, Y_j). $$
This yields the desired isomorphism and completes the proof.
\end{proof}

\begin{lemma} \label{sumprod}
Let $C$ be a thread category and let $\{R_i\}_{i\in I}$ be a family of objects in $\mathrm{rep}^{\mathrm{pwf}}(C)$. If either the product $\prod_{i\in I} R_i$ or the coproduct $\bigoplus_{i\in I} R_i$ belongs to $\mathrm{rep}^{\mathrm{pwf}}(C)$, then they both do, and we have an isomorphism
\[\prod_{i\in I} R_i \cong \bigoplus_{i\in I} R_i \]
in $\mathrm{rep}^{\mathrm{pwf}}(C)$.
\end{lemma}

\begin{proof}
Assume that the product $R \coloneqq \prod_{i\in I} R_i$ is in $\mathrm{rep}^{\mathrm{pwf}}(C)$; the other case is analogous. Since limits in functor categories are computed pointwise and $R \in \mathrm{rep}^{\mathrm{pwf}}(C)$, the vector space $R(x) \cong \prod_{i\in I} R_i(x)$ has finite dimension for every object $x \in C$. This forces $R_i(x)$ to be not $0$ for only finitely many indices $i$. Since direct products and direct sums of vector spaces coincide for finite families, we have:
\[ \prod_{i\in I} R_i(x) = \bigoplus_{i\in I} R_i(x). \]
Since both products and coproducts are computed pointwise, it immediately follows that $\prod_{i\in I} R_i \cong \bigoplus_{i\in I} R_i$ in $\mathrm{rep}^{\mathrm{pwf}}(C)$.
\end{proof}
\subsection{Proof of theorem \ref{hered}}
Let $X,Y \in \operatorname{rep}^{\mathrm{pwf}}(C)$, and let
\[
X \cong \bigoplus_{i\in I} X_i,
\qquad
Y \cong \bigoplus_{j\in J} Y_j
\]
be their Krull--Schmidt decompositions into indecomposable objects. The existence of such decomposition is guaranteed by Theorem \ref{KS}. We obtain the desired result via the following chain of isomorphisms:
\begin{align*}
\Ext^2_{\mathrm{rep}^{\mathrm{pwf}}(C)}(X,Y) &= \Ext^2_{\mathrm{rep}^{\mathrm{pwf}}(C)}\Big(\bigoplus_{i\in I}X_i, \bigoplus_{j\in J}Y_j\Big) \\
&\overset{(1)}{=} \Ext^2_{\mathrm{rep}^{\mathrm{pwf}}(C)}\Big(\bigoplus_{i\in I}X_i, \prod_{j\in J}Y_j\Big) \\
&\overset{(2)}{=} \prod_{i\in I}\prod_{j\in J} \Ext^2_{\mathrm{rep}^{\mathrm{pwf}}(C)}(X_i, Y_j) \\
&\overset{(3)}{=} \prod_{i\in I}\prod_{j\in J} \Ext^2_{\mathrm{rep}^{\mathrm{qnf}}(C)}(X_i, Y_j)\\
&\overset{(4)}{=} \prod_{i\in I}\prod_{j\in J} 0 \;=\; 0.
\end{align*}

Let us justify the numbered equalities:
\begin{enumerate}
    \item in $\mathrm{rep}^{\mathrm{pwf}}(C)$, the direct sum is equal to the direct product by Lemma \ref{sumprod};
    \item this equality is given by Lemma \ref{lemma ext};
     \item this equality follows from Theorem \ref{comb prop}, since the indecomposable summands $X_i$ and $Y_j$ belong to $\mathrm{rep}^{\mathrm{qnf}}(C)$;
      \item this equality is given by the hereditariness of ${\text{rep}^\text{qnf}}(C)$ (Corollary \ref{qnf hered}).
\end{enumerate}
Finally, since $\Ext^2_{\mathrm{rep}^{\mathrm{pwf}}(C)}(-,-) = 0$, the vanishing of the higher extension groups for all $n > 2$ follows immediately from Proposition \ref{prop:ext_vanish_general}. \qed

\end{document}